\documentclass[12pt]{amsart}

\usepackage{amsthm}
\usepackage{amsmath}
\usepackage{amsfonts}
\usepackage{amssymb}

\newtheorem{thm}{Theorem}[section]
\newtheorem{cor}[thm]{Corollary}
\newtheorem{lem}[thm]{Lemma}

\theoremstyle{definition}

\theoremstyle{remark}

\begin{document}

\title{Proof of  Kelly-Ulam Conjecture}

\author{Adel Tadayyonfar$^{\dagger}$ and Ali Reza Ashrafi$^*$  }

\thanks{$^*$ Corresponding author (Email: ashrafi@kashanu.ac.ir),}

\thanks{$^{\dagger}$ (Email: adeltadayyonfar@yahoo.com).}

\address{ Department of Pure Mathematics, Faculty of Mathematical Sciences, University of Kashan, Kashan 87317$-$53153, I. R. Iran}

\dedicatory{Dedicated to the Memory of Late Professor Maryam Mirzakhani}


\begin{abstract}
The deck of a graph $X$, $D(X)$, is defined as the multiset of all vertex-deleted subgraphs of $X$. Two graphs  are said  to be hypomorphic, if they have the same deck. Kelly-Ulam conjecture states that any two hypomorphic graphs on at least three vertices are isomorphic.

In this paper, we first prove that for two finite simple  hypomorphic graphs the number of $l$-paths between two arbitrary vertices are equal, where $1 \leq l \leq n - 2$. As a consequence, it is proved that the Kelly-Ulam  conjecture is correct over the category of all finite simple  graphs.

\vskip 3mm


\vskip 3mm

\end{abstract}

\maketitle


\section{Introduction}

Throughout this paper, all graphs are finite and simple. A vertex-deleted subgraph of a graph $X$ is a subgraph formed by deleting exactly one vertex  and all of its incident edges from $X$. The multiset of all vertex-deleted subgraphs of $X$ is called deck of $X$ and denoted by $D(X)$. Two graphs are called hypomorphic when they have the same deck. We use the notation $X \dot{\cong} Y$ to show that two graphs $X$ and $Y$ are hypomorphic. Moreover,  we denote a path of length $l$ (number of its edges) by $P_l$ and use the name $l$-path for such a path. For other concepts and notations not presented here, we refer to \cite{graph theory,distance}.

Kelly in 1957 \cite{Kellytree} showed that two hypomorphic trees are isomorphic. After him,  Ulam in 1960 asked the following question \cite{Ulam}:
\begin{quote}
``Suppose $A$ and $B$ are sets with $n$ elements for each ($n \geq 3$). A metric $\rho$ is given on $A$ with the property that $\rho(x,y)$ is either $1$ or $2$ whenever $x$ and $y$ are in $A$ and $x \neq y$. A similar metric is given on $B$. Now suppose that the $n - 1$ element subsets of $A$ and $B$ can be labeled, $A_1$, $\cdots$, $A_n$ and $B_1$, $\cdots$, $B_n$ in such a way that each $A_i$ is isometric to $B_i$. Does this force $A$ to be isometric to $B$?"
\end{quote}
In fact, Ulam asked that are two hypomorphic graphs isomorphic? Nash$-$Williams provides a counterexample to show that this conjecture is not true in infinite graphs and he discusses about kinds of recognizable infinite graphs (\cite{infinite}). Further, one has been found a lot of counterexamples for directed graphs, too. See \cite{directed} for details.

Throughout this paper, we assume that $n \geq 3$ and $G$, $H$ are two labeled graphs with vertex sets $\{v_1 , \cdots , v_n\}$ and $\{u_1 , \cdots , u_n\}$,  respectively, in such a way that for each $i$, $1 \leq i \leq n$, $G \setminus \{v_i\} \cong H \setminus \{u_i\}$. For simplicity, we assume that $G_i = G \setminus \{v_i\}$ and $H_i = H \setminus \{u_i\}$. Wall in her master dissertation \cite{Wall}, proved that if $G \dot{\cong} H$, then $|E(G)| = |E(H)|$ and their degree sequences are the same. Furthermore, she concluded that there is an isomorphism between two regular hypomorphic graphs. For the sake of completeness, we first give a  simple similar proof for some of the mentioned results.

\vskip 3mm

\begin{lem}\label{E}
For each $i$, $1 \leq i \leq n$, $|E(G)| = |E(H)|$ and $deg(v_i) = deg(u_i)$.
\end{lem}
\begin{proof}
Since for each $i$, $G_i \cong H_i$, $|E(G_i)| = |E(H_i)|$. In the following we denote this number with $E_i$. It is clear to see that
\begin{eqnarray}\label{EE}
|E(G)| = E_i + deg(v_i) \ \ &,& \ \ |E(H)| = E_i + deg(u_i).
\end{eqnarray}
Applying sigma on both sides of equalities in Equation (\ref{EE}) yields that $\sum_{i = 1}^n|E(G)|$ = $\sum_{i = 1}^nE_i$ + $\sum_{i = 1}^ndeg(v_i)$ and $\sum_{i = 1}^n|E(H)|$ = $\sum_{i = 1}^nE_i$ + $\sum_{i = 1}^ndeg(u_i)$. This proves that $n|E(G)|$ = $\sum_{i = 1}^nE_i$ + $2|E(G)|$ and $n|E(H)|$ = $\sum_{i = 1}^nE_i$ + $2|E(H)|$.  Therefore, $\sum_{i = 1}^nE_i$ = $(n - 2)|E(G)|$ and $\sum_{i = 1}^nE_i$ = $(n - 2)|E(H)|$.
The last relation shows that $|E(G)| = |E(H)|$ and so (1) implies that $deg(v_i) = deg(u_i)$.
\end{proof}

\vskip 3mm

\begin{cor}\label{degree sequences}
Assume that $G \dot{\cong} H$. Then
\begin{enumerate}
\item[a)] $G$ and $H$ have the same degree sequences.
\item[b)] $G$ is an Eulerian graph if and only if $H$ is Eulerian.
\end{enumerate}
\end{cor}
\begin{proof}
To prove (a), we apply Lemma \ref{E} and for part (b), we use the part (a)  and the fact that all vertices of an Eulerian graph have even degree \cite[Theorem 6.2.2.]{graph theory}.
\end{proof}

\vskip 3mm

In the next section, it is proved that the number of $l-$paths having $v_i$ and $u_i$ are equal. Also, one can see that the number of $l-$paths contains $v_i,v_j$ and $u_i,u_j$ are the same. Another conclusion is that two hypomorphic graphs are isomorphic.

\vskip 3mm


\section{Main Results}

\vskip 3mm

If $x$ and $y$ are two different vertices of an arbitrary graph $X$, then the number of paths of length $l$ in $X$ containing $x$ and both $x, y$ are denoted by $p_{X}(x , l)$ and $p_{X}(x , y, l)$, respectively.

\begin{lem}\label{implem}
Suppose $v_i$ and $v_j$ are distinct vertices in a graph $G$ and $1 \leq l \leq n - 2$. Then, $\sum_{i = 1 , i \neq j}^np_G(v_j , v_i , l)$ = $lp_G(v_j, l).$
\end{lem}
\begin{proof}
Define $A_j$ to be the  set  of  all  paths  of  $G$ of length  $l$  containing  $v_j$ and $A_{ji}$ to be the  set  of  all  paths of length  $l$  containing $v_j$ and $v_i$. Moreover,
\begin{eqnarray*}
B_j &:=& \{(a_j , \{v_{i_1} , \cdots , v_{i_l}\}) \ | \ a_j \in A_j \ , \ V(a_j) = \{v_j , v_{i_1} , \cdots , v_{i_l}\}\},\\
B_{ji} &:=& A_{ji} \times \{i\}.
\end{eqnarray*}

It is easy to see that if $i$, $j$ and $k$ are three distinct integers that $1 \leq i , j , k \leq n$, then $B_{ji} \cap B_{jk} = \emptyset$. This concludes that

\begin{eqnarray*}
\left|\bigcup_{i = 1 \ \& \ i \neq j}^nB_{ji}\right| &=& \sum_{i = 1 \ \& \ i \neq j}^n|B_{ji}|\\
&=& \sum_{i = 1 \ \& \ i \neq j}^n|A_{ji} \times \{i\}|\\
&=& \sum_{i = 1 \ \& \ i \neq j}^n|A_{ji}|\\
&=& \sum_{i = 1 \ \& \ i \neq j}^np_G(v_j , v_i , l).
\end{eqnarray*}
Moreover, one can easily observe that $|B_j| = |A_j| = p_G(v_j, l)$.
Each member of $B_j$ such as $(a_j , \{v_{i_1} , \cdots , v_{i_l}\})$ is corresponding one-to-one to an element of $B_{ji_r}$, $1 \leq r \leq l$. So, $\bigcup_{i = 1,  i \neq j}^nB_{ji}$ has $l$ copies of every member of $B_j$. Therefore, $l|B_j|$ = $|\bigcup_{i = 1, i \neq j}^nB_{ji}|$. This implies that $p_G(v_j , l)$ = $\frac{\sum_{i = 1, i \neq j}^np_G(v_j , v_i , l)}{l}$, which proves the Lemma.
\end{proof}

\vskip 3mm

For a graph $X$, a cutnode of $X$ is a vertex whose removal increases the number of components. Moreover,  a nonseparable graph is connected, without any cutnode and not to be a single point. A block of $X$ is a maximal nonseparable subgraph. For an arbitrary cutnode $x \in V(X)$, we mean $bl(x)$ by the number of blocks connected to $x$. If $x$ is not a cutnode, then we define $bl(x) = 1$.

\vskip 3mm

\begin{lem}\label{cutnodes}
Let $G$ and $H$ be two connected hypomorphic graphs. Then
\begin{enumerate}
\item[a)] $v_i$ is a cutnode if and only if $u_i$ is a cutnode.
\item[b)] $bl(v_i) = bl(u_i)$.
\end{enumerate}
\end{lem}

\begin{proof}

(a) \ Suppose that $v_i$ is a cutnode of $G$. So, $G_i$ is disconnected. If $u_i$ is not a cutnode of $H$, then $H_i$ is connected which is impossible. The converse  is similar. (b) \ Since $bl(v_i)$ is exactly the number of connected components of $G_i$ and  $G_i \cong H_i$, $bl(v_i) = bl(u_i)$, proving the lemma.
\end{proof}

\vskip 3mm

Harary in \cite{connected} demonstrated that connected  is reconstructible. For a graph $X$, the number of connected components of $X$ is denoted by $c(X)$. It is trivial to see that if $X$ is connected, then $c(X) = 1$. Now, we are ready to prove that the number of connected components of two hypomorphic graphs are the same.

\vskip 3mm

\begin{thm}
If $G \dot{\cong} H$, then $c(G) = c(H)$.
\end{thm}

\begin{proof}
For each $i$, $1 \leq i \leq n$, $c(G_i) = c(G) + bl(v_i) - 1$. Now, applying sigma on both sides of the equation yields that $\sum_{i = 1}^nc(G_i)$ = $\sum_{i = 1}^nc(G)$ + $\sum_{i = 1}^nbl(v_i)$ $-$ $\sum_{i = 1}^n1$ which means that $\sum_{i = 1}^nc(H_i)$ = $nc(G)$ + $\sum_{i = 1}^nbl(u_i)$ $-$ $n$, by part (b) of Lemma \ref{cutnodes}. This shows that $nc(G) = nc(H)$ and thus, the number of connected components of $G$ and $H$ are the same.
\end{proof}

\vskip 3mm

In Theorem \ref{p}, we are showing that two hypomorphic graphs have the same number of paths of length $l$.

\vskip 3mm

\begin{thm}\label{p}
For two hypomorphic graphs $G$ and $H$, $p_G(v_j,l) = p_H(u_j,l)$ and $p_G(v_i, v_j, l) = p_H(u_i, u_j, l)$, where  $1 \leq i \neq j \leq n$ and $1 \leq l \leq n - 2$.
\end{thm}
\begin{proof}
Since for each $i$, $G_i \cong H_i$, $p_{G_i}(v_j , l) = p_{H_i}(u_j , l)$ in which $i \neq j$. It is clear to see that
\begin{enumerate}
\item[(2)]  $p_G(v_j , l) = p_{G_i}(v_j , l) + p_G(v_j , v_i , l),$
\item[(3)] $p_H(u_j , l) = p_{H_i}(u_j , l) + p_H(u_j , u_i , l).$
\end{enumerate}
Applying sigma on both sides of equals (2) yields that $\sum_{\stackrel{i = 1}{i \neq j}}^np_G(v_j , l)$ = $\sum_{\stackrel{i = 1}{i \neq j}}^np_{G_i}(v_j , l)$ + $\sum_{\stackrel{i = 1}{i \neq j}}^np_G(v_j , v_i , l)$. This implies that $(n - 1)p_G(v_j,l)$ = $\sum_{\stackrel{i = 1}{i \neq j}}^np_{G_i}(v_j,l)$ + $lp_G(v_j,l)$, by Lemma \ref{implem}. Therefore, $\sum_{\stackrel{i = 1}{i \neq j}}^np_{G_i}(v_j,l)$ = $(n - l - 1)p_G(v_j,l)$.
A similar argument as above and Equation (3) show that in $H$, $\sum_{\stackrel{i = 1}{i \neq j}}^np_{H_i}(u_j , l)$ = $(n - l - 1)p_H(u_j , l)$ which immediately leads to $p_G(v_j,l) = p_H(u_j,l)$, for $1 \leq l \leq n - 2$. We now apply (2) and (3) to prove that $p_G(v_i,v_j,l)$ = $p_H(u_i,u_j,l)$. This completes the proof.
\end{proof}

\vskip 3mm

The next corollary proves the Kelly-Ulam conjecture  in the category of simple and finite graphs.

\vskip 3mm

\begin{cor}
If $G \dot{\cong} H$, then $G \cong H$.
\end{cor}
\begin{proof}
Suppose $G$ and $H$ are two simple and finite hypomorphic graphs. By putting $l = 1$ in Theorem \ref{p}, it is observed that $p_G(v_i,v_j,1) = p_H(u_i, u_j, 1)$ for all $i, j$, $1 \leq i \neq j \leq n$. This means that $v_iv_j \in E(G)$ if and only if $u_iu_j \in E(H)$, which concludes that $G$ = $H$.
\end{proof}

\vskip 3mm


\noindent{\bf Acknowledgement.} The research of the second author is partially supported by the University of Kashan under grant no. 364988/66.


\end{document}